\documentclass[a4paper,11pt]{article}
\usepackage{makeidx}
\usepackage[mathscr]{eucal}
\usepackage{amssymb, amsfonts, amsmath, amsthm, graphics, color}

\newtheorem{definition}{Definition}[section]
\newtheorem{proposition}{Proposition}[section]
\newtheorem{theorem}{Theorem}[section]
\newtheorem{lemma}[proposition]{Lemma}
\newtheorem{remark}{Remark}[section]
\newtheorem{corollary}[theorem]{Corollary}

\numberwithin{equation}{section}

\title{ \textbf{Ground State Solutions of the Complex Gross Pitaevskii Equation Associated to Exciton-Polariton Bose-Einstein Condensates}}
\author{Hichem Hajaiej$^{\text{(1)}}$ \ Slim Ibrahim$^{\text{(2)}}$  
 \& Nader Masmoudi$^{\text{(3)}}$}

\newcommand{\Addresses}{{
  \bigskip
  \footnotesize

  Hajaiej~H. \textsc{California State University, Los Angeles, 5151 University Drive, 90032, Los Angeles, CA}\par\nopagebreak
  \textit{E-mail address}, Hajaiej~H.: \texttt{hhajaie@calstatela.edu}

  \medskip

  Ibrahim~S. \textsc{Department of Mathematics and Statistics, University of Victoria, PO Box 3060 STN CSC, Victoria, BC, V8P 5C3, Canada}\par\nopagebreak
  \textit{E-mail address}, Ibrahim~S.: \texttt{ibrahims@uvic.ca}

  \medskip

  Masmoudi~N. \textsc{New York University, The Courant Institute for Mathematical Sciences}\par\nopagebreak
  \textit{E-mail address}, Masmoudi, N.: \texttt{masmoudi@courant.nyu.edu}

}}

\date{}
\begin{document}
\maketitle
\textbf{Abstract:} We investigate the existence of ground state solutions of a Gross-Pitaevskii equation modeling the dynamics of pumped Bose Einstein condensates (BEC). The main interest in such BEC comes from its important nature as macroscopic quantum system, constituting an excellent alternative to the classical condensates which are hard to realize because of the very low temperature required. Nevertheless, the Gross Pitaevskii equation governing the new condensates presents some mathematical challenges due to the presence of the pumping and damping terms. Following a self-contained approach, we prove the existence of ground state solutions of this equation under suitable assumptions: This is equivalent to say that condensation occurs in these situations. We also solve the Cauchy problem of the nonlinear Schr\"{o}dinger equation and prove some corresponding laws.

\section{Introduction}
The first realization of condensation has been obtained experimentally in a system consisting of about half million alkali atoms cooled down to nano-Kelvin  temperature. Thus, a considerable obstacle in the study of (BEC) is the very low temperature required to create the condensate. Completely aware that it is extremely important to explore what kind of condensates can undergo condensation at higher temperatures, huge efforts have been undertaken by scientists to overcome this difficulty right after the first experimental realization of the first  (BEC) in 1995. During the last years, a new kind of condensates has attracted the attention of many scientists. Very recently, it turned out that an excellent candidate is a system of exciton-polaritons, which are bosonic quasiparticles that exist inside semiconductor micro-cavities, consisting of a superposition of an exciton and a cavity photon. Above a threshold density, the polaritons macroscopically occupy the same quantum state, forming a condensate. The temperatures that are usually used to form exciton-polariton BECs are around T=10K, far higher than the nano-Kelvin temperatures required for atomic BECs. They are immensely promising in terms of new quantum technologies since quantum effects can appear on a macroscopic level, unlike most systems where quantum effects are rather easily destroyed by temperature and decoherence.  As Boson particles are composed of quantum well excitons and optical cavity photons, microcavity exciton-polaritons possess unique intrinsic features: reminiscent excitonic nature leads to important interaction dynamics among exciton-polaritons.  Polariton-polariton repulsive interactions are indeed crucial to stimulate scattering processes in order to relax into the ground state Bose-Einstein condensates (BECs).  Since the temperature of condensation is inversely proportional to the mass of the particles, the exciton-polariton systems afford relatively high temperatures of condensation. The first drawback of these new condensates is their very short lifetime (approximately 1 ps), inherited also from their photonic component, so that polariton thermalization could be problematic. In fact the polariton gas can become fully thermalized, as a result of strong polariton-polariton interaction caused by their excitonic component.  The second important inconvenient comes from the fact that the excitons disappear with the recombination of the electron-hole pairs through emission of photons. One way to overcome these problems is to introduce a polariton reservoir: polaritons are \lq\lq cooled\rq\rq and \lq\lq pumped\rq\rq from this reservoir into the condensate. At the same time, a low density level is kept in order to reduce the interactions between polaritons. Different mathematical models have been suggested for this new condensate. In this paper we consider the one proposed in \cite{KB}, called complex Gross-Pitaevski equation. For a more detailed account of these aspects, see \cite{SKMW} and references therein.

In \cite{SKMW}, the authors addressed the nature of radially symmetric standing wave-type solutions of  the following nonlinear Gross-Pitaevskii equation:
\begin{equation}\label{NLS}\tag{GPPD}
i\displaystyle{\frac{\partial \psi}{\partial t}}=(-\Delta+V(x)+|\psi|^{2})\psi+i(\sigma(x)-\alpha|\psi|^{2}) \psi,
\end{equation}
where $\psi=\psi(x,t)$ is a complex-valued function defined on $\mathbb{R}^{2}\times \mathbb{R}$,
$\Delta$ is the Laplace operator on $\mathbb{R}^{2}$, $V(x)=|x|^2$ is the harmonic potential, $\sigma(x)\geq0$ and $\alpha\geq0$.

To achieve their goals, they have developed a numerical collocation method but they did not provide any theoretical justification of their claims. The main objective of this paper is to rigorously prove the existence of ground state solutions of the Gross-Pitaevskii equation under study. We believe that this is a challenging and immensely important scientific question. The principle challenge comes from the fact that all classical methods do not seem to be applicable to discuss the existence of stationary solutions to (GPPD). This is essentially due to the simultanous presence of the dissipation and pumping terms simultaneously. Let us note that the establishment of ground state solutions avoids costly and very difficult experiments in the \lq\lq classical\rq\rq BEC. To achieve this goal, let us first introduce some important quantities associated to  (GPPD).

Recall that the mass $\mathcal{M}$, the Hamiltonian $\mathcal{H}$, the action $\mathcal{S}_{\mu}$ ($\mu>0)$ and the functional $\mathcal K$
associated to the equation \eqref{NLS} are given by:
\begin{align}
\label{Mu}
\mathcal{M}(u)&:=\left\| u \right\|_{L^{2}}^{2},
\\[6pt]
\label{Hu}
\mathcal{H}(u)
&:=
\frac{1}{2}(\left\|\nabla u \right\|_{L^{2}}^{2}+\left\|x u \right\|_{L^{2}}^{2})
+
\frac{1}{4}\left\| u \right\|_{L^{4}}^{4}:=\mathcal H_0(u)+\frac{1}{4}\left\| u \right\|_{L^{4}}^{4}
\\[6pt]
\label{Su}
\mathcal{S}_{\mu}(u)&:=-\frac{\mu}{2} \mathcal{M}(u)+\mathcal{H}(u),
\\[6pt]
\label{Ku}
\mathcal K(u)
&:=\int_{\mathbb R^2}(\sigma(x)-\alpha|u(x)|^2)|u(x)|^2\;dx,
\end{align}
respectively. Observe that
\begin{equation}\label{ddt mass}
\frac d{dt}{\mathcal M}(\psi(t))=\mathcal K(\psi(t))
\end{equation}
and

\begin{equation}\label{ddt hamil}
\frac d{dt}\mathcal H(\psi(t))=\int_{\mathbb R^2}(\sigma-\alpha|\psi|^2)(|\psi|^4+V|\psi|^2+|\nabla\psi|^2)-
2\alpha{\mathcal (R(\psi\nabla\bar\psi)})^2\;dx.
\end{equation}
Identity \eqref{ddt mass} shows that, at least formally, the mass and the energy are pumped into the system through the term $i\sigma\psi$ involving the parameter $\sigma$ and they are
nonlinearly damped by the term $-i\alpha|\psi|^2\psi$ involving the parameter $\alpha$. Contrarily to the complex Ginzburg-Landau equation (when a dissipatif term of the form $i\Delta\psi$ is added to the RHS of (GPDP)), one cannot obtain time-uniform estimates of the solution in the energy space. The complex Gross-Pitaevski equation reflects the non-equilibrium dynamics described above by adding pumping and decaying terms to the GP equation.\\
Before going any further, we recall a few results about the linear equation without dissipation and pumping. The equation then reads
$$
i\displaystyle{\frac{\partial \phi}{\partial t}}=(-\Delta+V(x))\phi.
$$
We define the energy space $\Sigma:=H^{1}(\mathbb{R}^{2})\cap\{u:\;xu\in L^2\}$, endowed with the $L^2$-scalar product $(u,v)_2:=\int_{\mathbb{R}^2}u(x)\bar v(x)\;dx$, by

$$
(u,v)_{\Sigma}=(\nabla u,\nabla v)_2+(x u,x v)_2+(u,v)_2: \left\| u \right\|_{\Sigma}^{2}=\left\|\nabla u \right\|_{2}^{2}+\left\|{(1+(|\cdot|^2)^{\frac{1}{2}} u} \right\|_{2}^{2}.
$$
Also, define the dual space $\Sigma^*$ of $\Sigma$ as follows. For any $v\in\Sigma^*$, there exists a unique $u\in\Sigma$ such that $H_0u=v$
with the norm on $\Sigma^*$ given by

$$
\|H_0u\|_{\Sigma^*}=\|v\|_{\Sigma^*}:=\|u\|_\Sigma.
$$
Recall that $\left\| \cdot\right\|_{p} $ is the norm in $ L^p(\mathbb R^2)$. It is well known that the
unbounded operator $H_0:= -\Delta +V$ defined on
$$
D(H_0):=\{u\in\Sigma:\quad H_0(u)\in L^2(\mathbb R^2)\}
$$
is self-adjoint.  Moreover, the lowest eigenvalue of $H_0$ denoted by $\omega_1=2$ is simple with eigenfunction
$\varphi_1(x)=\frac1{\sqrt\pi} e^{-|x|^2/2}$.
Notice that $(\varphi_1,\omega_1)$ can be constructed variationally as

$$
\omega_1=\min_{\|u\|_{L^2}=1}\frac12\int_{\mathbb R^2}|\nabla u|^2+|x|^2|u|^2\;dx:=\min _{\|u\|_{L^2}=1}\mathcal H_0.
$$
In particular, for any $u\in D(H_0)$, we have

$$
2\|u\|_{L^2}^2\leq \|xu\|_{L^2}^2+\|\nabla u\|_{L^2}^2.
$$
For more details, we refer for example to \cite{KavianWeissler}.

When the chemical potential is complex $\mu=\mu_r+i\mu_i$, solitary wave solution $\psi(x,t)=Q(x)e^{-it\mu}=Q(x)e^{t\mu_i}e^{-it\mu_r}$
would grow exponentially fast as $|t|\to\infty$ which can be bad for the analysis as well as for numerics and experiments. Assuming that $\mu=\mu_r$, yields the following stationary problem for $Q$:
\begin{equation}\label{mu SP}\tag{$\mu$-${\rm SP}$}
\mu Q= (-\Delta+V(x)+|Q|^{2})Q+i(\sigma(x)-\alpha|Q|^{2}) Q,
\qquad
Q \in \Sigma\setminus\{0\} .
\end{equation}
Multiplying \eqref{mu SP} by $\bar Q$ and integrating gives the following identity.
$$
\mu \mathcal M(Q)=2 \mathcal H(Q)+1/2\|Q\|_{L^{4}}
^4+i \mathcal K(Q).
$$
The condition for the chemical potential $\mu$ of being real is then equivalent to the fact that $Q$ is a zero of $\mathcal K$.

It is important to emphasize that due to the presence of the dissipation and pumping mechanisms, we find it hard to apply the standard variational or PDE
methods to construct soliton-type solutions of (GPPD) (i.e. a solution $Q$ of \eqref{mu SP}).
In this paper, our idea to construct a solution of \eqref{mu SP} with real chemical potential $\mu$ goes along a perturbative
way by introducing a small parameter factor in the dissipation and pumping term.
More precisely, for all $\varepsilon>0$, consider

\begin{equation}\label{NLS eps}\tag{${\rm GPPD}_ \varepsilon$}
i\displaystyle{\frac{\partial \psi}{\partial t}}=(-\Delta+V(x)+|\psi|^{2})\psi+i\varepsilon(\sigma(x)-\alpha|\psi|^{2}) \psi,
\end{equation}
and  its corresponding stationary equation

\begin{equation}\label{mu SP eps}\tag{$\mu$-${\rm SP}_\varepsilon$}
\mu Q= (-\Delta+V(x)+|Q|^{2})Q+i\varepsilon(\sigma(x)-\alpha|Q|^{2}) Q
\qquad
Q \in \Sigma\setminus\{0\} .
\end{equation}
The object is to construct a solution $(Q_\varepsilon,\mu_\varepsilon)$ in the form

$$
Q_\varepsilon=Q_\varepsilon^a+\psi_\varepsilon,\quad\mbox{and}\quad \mu_\varepsilon=\mu_\varepsilon^a+\mu_\varepsilon,
$$
where the approximate solution $(Q_\varepsilon^a,\mu_\varepsilon^a)$ will be given explicitly, and $(\psi_\varepsilon,\mu_\varepsilon)$ is
the error term that needs to be found. To define $(Q_\varepsilon^a,\mu_\varepsilon^a)$, we need to introduce some notation and state a few preliminary useful results. The first Theorem  of this paper reads as follows:

\begin{theorem}
\label{Solit wave}
Let $\sigma(x)\geq0$ be a continuous nontrivial function. There exist $\alpha_0\gg1$ and a positive $\varepsilon_0$ small such that, for any $0<\varepsilon<\varepsilon_0$ and  $\alpha>\alpha_0$, the complex Gross-Pitaevki{\color{green}{i}} equation \eqref{NLS eps} has a solitary wave solution $\psi^\varepsilon(x,t)=e^{it\mu_\varepsilon}Q^\varepsilon(x)$
with $(Q,\mu_\varepsilon)\in\Sigma\times(2,\infty)$ solving \ref{mu SP eps}.
\end{theorem}

\begin{remark}
It would be very desirable to extend the branch of standing wave solutions we constructed for $\varepsilon$ small to all values of $\varepsilon$. Unfortunately, so far we were not able to do so given the non-equilibrium structure of the model.
\end{remark}
Our second result concerns the Cauchy problem associated to \eqref{NLS}. We have.

\begin{theorem}
\label{Cauchy pb}
Assume $\alpha\geq0$, and $\sigma\in L^\infty(\mathbb R^2)\cap L^4(\mathbb R^2)$. For any $\psi_0\in L^2(\mathbb R^2)$, there exists a unique global solution $\psi\in C([0,\infty),L^2(\mathbb R^2))\cap L^4_{\text{loc}}([0,\infty),L^4(\mathbb R^2))$ of \eqref{NLS} with $\psi(x,0)=\psi_0(x)$.  Moreover, for any $T>0$, we have

$$
\int_0^T\int_{\mathbb R^2}|\psi(x,t)|^4\;dxdt\lesssim e^{\|\sigma\|_{L^\infty}T}\|\psi_0\|_{L^2}.
$$
\end{theorem}

The paper is organized as follows: In the next section, some preliminary results are proven. This will prepare the field to the establishment of ground state solutions. In section 3, we will present our self-contained proof built up to prove the existence of ground state solutions. The last section of this paper is dedicated to the Cauchy problem. We show the existence and uniqueness of solutions for a large class of damping and pumping terms. We also discuss the non-conservation of some important functionals associated to the Schr\"{o}dinger equation.

\section{Preliminaries}
Here we focus on the problem without pumping and decay of the energy, that is when $\varepsilon=0$. We start by recalling a few known facts about the space $\Sigma$, for which the proof can for example be found in Kavian-Weissler \cite{KavianWeissler}.
\begin{lemma}\label{lemma 1}
The Hilbert space $\Sigma$ is compactly embedded in $L^p(\mathbb{R}^2)$ for any $p\in[2,\infty)$.
\end{lemma}
Throughout this paper, we suppose that $\sigma\geq0$ is nontrivial continuous and is in  $L^\infty(\mathbb{R}^2)$ function.

\begin{lemma}\label{lemma 2}
For any $M>0$, there exists a unique $v_M\in\Sigma$ solving the following constrained variational problem:
$$(V_M):\quad \mu_M=\inf\{\mathcal H(u):\int u^2=M\}; $$
In addition, $v_M$ is non-negative, radial and radially decreasing.
\end{lemma}
\begin{proof} It is sufficient to show  the existence of a minimizer of $(V_M)$. The uniqueness of the minimizer follows
directly from the strict convexity of the functional $\mathcal H$.

Now let us fix $M>0$, let $(v_n)$ be a minimizing sequence of $(V_M)$, i.e., $\lim_{n\to\infty}\mathcal H(v_n)=\mu_M$
and $\int v^2_n=M$.
Then $$\mathcal H(v_n)\geq\frac{1}{2}\|\nabla v_n\|_2^2+\frac{1}{2}\|x v_n\|_2^2.$$
Therefore, we can find $K_M>0$ such that
$$\|\nabla v_n\|_2^2+\|x v_n\|_2^2\leq K_M.$$
This implies that
\begin{equation}\label{eq a}
\|v_n\|_{\Sigma}^2\leq M+K_M.
\end{equation}
Consequently, there exists $u\in\Sigma$ such that
$$v_n\rightharpoonup u\quad {\rm in} \ \ \ \Sigma.$$
This implies, thanks to Lemma \ref{lemma 1}, that $v_n\to u$ in $L^2(\mathbb{R}^2)$ and $L^4(\mathbb{R}^2)$.
Thus, we certainly have that $\int u^2=M$ implying that $u$ is non-trivial, and by the lower semi-continuity, we can write:
$$\mathcal H(u)\leq \liminf_n \mathcal H(v_n)=\mu_M.$$
Therefore, $\mathcal H(u)=\mu_M$.
On the other hand, let $u$ be the unique minimizer of $(V_M)$, then $u$ is a non-negative function in $\Sigma$ since
$$
\mathcal H(|u|)\leq \mathcal H(u), \quad\mbox{and}\quad M(|u|)=M(u).
$$
Furthermore, by rearrangement inequalities \cite{HS1,HK}, we have:
$$\int {|u|}^2=\int (|u|^*)^2$$
$$\int {|u|}^4=\int (|u|^*)^4$$
$$\int {|x|^2}{|u|}^2\geq\int |x|^2(|u|^*)^2$$
$$\int |{\nabla|u|}|^2\geq\int |\nabla(|u|^*)|^2.$$
Combining these identities, it follows that
$$\mathcal H(|u|^*)\leq \mathcal H(|u|).$$
\end{proof}
The next Lemma, addresses the regularity of the Hamiltonian $\mathcal H$, as well as the map  $M\to\mu_M$.

\begin{lemma}\label{lemma 3}
\smallskip
The Hamiltonian $\mathcal H$ is in $C^1(\Sigma,\mathbb{R})$. Moreover, for all $u\in\Sigma$ we have
\begin{equation}
\tag{i} \|\mathcal H'(u)\|_{\Sigma^{-1}}\leq C\{\|u\|_{\Sigma}+\|u\|_{\Sigma}^3\}\quad \mbox{for all}\quad  u \in\Sigma,
\end{equation}
and the function
\begin{equation} \tag{ii}  M\to \mu_M=\mathcal H(v_M),\quad \mbox{is continuous on}\quad (0,\infty).
\end{equation}
\end{lemma}

\begin{proof} The proof of (i)  follows from standard arguments. For example, we refer to reference \cite{HS2}, and we just prove (ii).\\
Fix $M>0$. Let $M_n\subset(0,\infty)$ be a sequence of positive real numbers such that $M_n\to M$. We will first
prove that
\begin{equation}\label{eq 2}
\limsup_n \mu_{M_n}\leq \mu_M.
\end{equation}
Let $(v_n)$ be a sequence such that $\int v_n^2=M$ and $\mathcal H(v_n)\to \mu_M.$
By (\ref{eq a}), we can find $L>0$ such that
$$\|v_n\|_{\Sigma}^2\leq L.$$
Now let $w_n=\frac{M_n}{M}v_n$, then $\int w_n^2=M_n$ and
$$\|v_n-w_n\|_{\Sigma}=|1-\frac{M_n}{M}|\|v_n\|_{\Sigma}\leq |1-\frac{M_n}{M}|L$$
for any $n\in \mathbb{N}$.

Therefore, we can find $n_0$ such that
$$\|v_n-w_n\|_{\Sigma}\leq L+1$$
for any $n\geq n_0$.

It follows from  (i) that there exists  a constant $K(L)$ such that $\|\mathcal H'(u)\|_{\Sigma^{-1}}\leq K(L)$
for all $u\in \Sigma$ such that $\|u\|_{\Sigma}\leq 2L+1$.

Thus, for all $n\geq n_0$,
\begin{eqnarray*}
|\mathcal H(w_n)-\mathcal H(v_n)|&=& |\int_0^1\frac{d}{dt}\mathcal H(tw_n+(1-t)v_n)dt|
\\&\leq& \sup_{\|u\|_{\Sigma}\leq 2L+1}\|\mathcal H'(u)\|_{\Sigma^{-1}}\|v_n-w_n\|_{\Sigma}
\\&\leq& K(L)L|1-\frac{M_n}{M}|.
\end{eqnarray*}
Consequently, $\mu_{M_n}\leq \mathcal H(w_n)\leq \mathcal H(v_n)+ K(L)L|1-\frac{M_n}{M}|.$

Then $\limsup \mu_{M_n}\leq \lim \mathcal H(v_n)=\mu_M$ and then
\begin{equation}\label{limsup}
\limsup \mu_{M_n}\leq \mu_M.
\end{equation}
Now let us prove that if $M_n\to M$, then
\begin{equation}\label{eq 4}
\mu_M\leq \liminf \mu_{M_n}.
\end{equation}
For all $n\in \mathbb{N}$, there exists $(v_n)$ a sequence of functions in $\Sigma$ such that $\int v_n^2=M_n$ and
$$\mu_{M_n}\leq \mathcal H(v_n)\leq \mu_{M_n}+\frac{1}{n}.$$
Combining the proof of (\ref{eq a}) and (\ref{eq 4}), we can find $K>0$ such that $\|v_n\|_{\Sigma}\leq K$ for all $n\in \mathbb{N}$.
Setting $w_n=\frac{M}{M_n}v_n$, we have that $\int w_n^2=M$ and $$\|v_n-w_n\|_{\Sigma}\leq K|1-\frac{M}{M_n}|.$$
Thus,  following the proof of (\ref{eq 4}), we certainly get:
$$|\mathcal H(w_n)-\mathcal H(v_n)|\leq L(K)K|1-\frac{M}{M_n}|.$$
Consequently, we have:
$$\mu_{M_n}\geq \mathcal H(v_n)-\frac{1}{n}\geq \mathcal H(w_n)-L(K)K|1-\frac{M}{M_n}|-\frac{1}{n},
$$
yielding $\liminf \mu_{M_n}\geq \mu_M$ as desired.
\end{proof}

\begin{proposition}\label{prop 1}
Let $M>0$, and $(M_n)\subset(0,\infty)$ be a sequence of positive real numbers such that $M_n\to M$. Denote by $v_{M_n}$
the unique minimizer of $(V_{M_n})$, and $v_M$ the unique minimizer of $(V_M)$. Then
$$
\mathcal K(v_{M_n})\to \mathcal K(v_M),
$$
and
$$
\mathcal H(v_{M_n})\to \mathcal H(v_M).
$$
\end{proposition}
\begin{proof}
We will first prove that there exists $\bar u \in \Sigma$ such that $v_{M_n}$ converges weakly in $\Sigma$ to $\bar u$ $(v_{M_n}\rightharpoonup \bar u\ \ in \ \ \Sigma)$.
First obviously $\|v_{M_n}\|_2^2\leq A$. Now noticing that
$$
\mu_{M_n}=\frac{1}{2}\|\nabla v_{M_n}\|_2^2+\frac{1}{2}\|x v_{M_n}\|_2^2+\frac{1}{4}\| v_{M_n}\|_4^2,
$$
one has
$$
\mu_{M_n}\geq\frac{1}{2}\|\nabla v_{M_n}\|_2^2+\frac{1}{2}\|x v_{M_n}\|_2^2.
$$
Therefore, using (\ref{eq 4}), there exists a constant $B>0$ such that $$\|v_{M_n}\|_{\Sigma}\leq B.$$
Thus, (up to a subsequence), there exists $\bar u\in\Sigma$ such that
$$v_{M_n}\rightharpoonup \bar u\ \ \  in\ \ \Sigma.$$
Now using Lemma \ref{lemma 1}, we have that
$$v_{M_n}\to \bar u\ \ \  in\ \ L^2(\mathbb{R}^2)\cap L^4(\mathbb{R}^2).$$
In particular, $\int\bar u^2=M$. Thus,
$$\mu_M\leq \mathcal H(\bar u)\leq \liminf \mathcal H(v_{M_n})=\liminf\mu_{M_n}$$
and then $\mathcal H(\bar u)=\mu_M$.
This shows that $\bar u$ is the unique minimizer of $(V_M)$. To end the proof, we need to show that
\begin{equation}\label{eq 51}
\int\sigma(x)v_{M_n}^2(x)\to \int\sigma(x)v_{M}^2(x)
\end{equation}
and
\begin{equation}\label{eq 52}
\int v_{M_n}^4(x)\to \int v_{M}^4(x).
\end{equation}
To prove (\ref{eq 51}), it is sufficient to notice that $\sigma\in L^\infty(\mathbb{R}^2)$ and $v_n\to v$ in $L^2(\mathbb{R}^2)$, while (\ref{eq 52}) follows from the fact that
$v_n\to u$ in $L^4(\mathbb{R}^2)$.
\end{proof}

\section{Ground State Solutions}
Always in the case $\varepsilon=0$, and within the class of minimizers $v_M$ we have just constructed, we would like to intersect it with the co-dimension one manifold characterized by the zeros of the functional $\mathcal K$. Before doing so, let us first fix our assumptions on the decay and pumping parameters.\\
First we deal with case $\varepsilon=0$ i.e. the standard nonlinear Schr\"odinger equation in the absence of both the
pumping and dissipation. Equation \eqref{mu SP eps} then becomes

\begin{equation}\label{mu SP eps 0}\tag{$\mu$-${\rm SP}_0$}
\mu Q= (-\Delta+V(x)+|Q|^{2})Q,
\qquad
Q \in \Sigma\setminus\{0\} .
\end{equation}
The first preliminary result is the first iteration. We have the following result:

\begin{proposition}
\label{exis Q0}
There exists a non-negative radial function $Q_0\in \Sigma$ and $\mu_0>2$ solving \eqref{mu SP eps 0}. Moreover, $Q_0$ satisfies
$$
\mathcal K(Q_0)=0.
$$
\end{proposition}
\begin{remark}
$(Q_0,\mu_0)$ will be the first approximate solution in the iteration process to construct the full solution
$(Q_\varepsilon,\mu_\varepsilon)$ of \eqref{mu SP eps}.
\end{remark}
\begin{proof}[Proof of Proposition \ref{exis Q0}]
It is sufficient to prove that the functional $\mathcal K$ changes sign when the mass of the ground state $v_M$ given by Lemma \ref{lemma 2} varies. Then the conclusion will follow using Lemma \ref{lemma 3}. Now, because of the positivity of $v_M$, first observe that for any nontrivial non-negative continuous function $\sigma$, we have $\int_{\mathbb R^2}\sigma |v_M|^2\;dx>0$. Moreover, on the one hand, by the Gagliardo-Nirenberg inequality, there is a constant
$C_*>0$ such that for any $u\in H^1$, we have
$$
\|u\|_{L^4}^4\leq C_*\|\nabla u\|_{L^2}^2\|u\|_{L^2}^2.
$$
On the other hand, multiplying \eqref{mu SP eps 0} by $\bar u$ and integrating shows that any solution $u$ of \eqref{mu SP eps 0} satisfies
$$
\mu\|u\|_{L^2}^2=\|\nabla u\|_{L^2}^2+\|xu\|_{L^2}^2+\|u\|_{L^4}^4.
$$
Thus, if $\|u\|_{L^2}^2=M$ we have
$$
\|u\|_{L^4}^4\lesssim M^2\mu_M.
$$
This shows that when $M\leq 1$, we have $\mu_M\lesssim1$ and thus $\mathcal K(u_M)\geq \int_{\mathbb R^2}\sigma |u|^2\;dx-CM^2$, for some positive constant $C$.
Now since $\sigma\geq0$ is a nontrivial continuous function, there exists a nontrivial open set  $\mathcal O\subset\mathbb R^2$ and a positive constant $c_0>0$ such that $\sigma(x)\geq c_0 $, for all $x\in\mathcal O$. We have $\int_{\mathbb R^2}\sigma |u|^2\;dx\geq c_0 \int_\mathcal O|u|^2\;dx\geq c_1M$, for some small positive constant $c_1$. This implies that $\mathcal K(u_M)\geq 0$ as $M\to0$.
Now, we  need to show that $\mathcal K(u_M)$ becomes negative for large masses.
In fact, first we will  prove  that
\begin{eqnarray}
\label{crucial up bnd}
\mathcal H(u_M) \lesssim M^\frac32,\quad\mbox{as}\quad M\to\infty.
\end{eqnarray}

If we let $\mathcal H_{\text{int}}(u):=\frac12(\|x u\|_{L^2}^2+\frac12\|u\|_{L^4}^4)$, then clearly

$$
\mathcal H_{\text{int}}(u_M)\leq \mathcal H(u_M).
$$
Now, we will explicitly calculate

$$
\nu_M:=\inf _{\|u\|_{L^2}^2=M}\mathcal H_{\text{int}}(u), \ \ u\in\Sigma_{\text{int}},
$$
where $\Sigma_{\text{int}}=\{u\in L^2(\mathbb{R}^2), u\in L^4(\mathbb{R}^2): \int|x|^2u^2<\infty\}$ with the norm
$$\|u\|_{\Sigma_2^4}=\|u\|_2+\|u\|_4+\||x|u\|_2.$$
Let $(u_n)$ be a minimizing sequence of $\nu_M$ that is
\begin{eqnarray}\label{minimiz seq}
\|u_n\|_{L^2}^2=M,\quad\mbox{and} \quad\frac12(\|x u_n\|_{L^2}^2+\frac12\|u_n\|_{L^4}^4)\to\nu_M.
\end{eqnarray}
From the above bounds, let us just denote by $u$ (instead of $u_M$), an $L^2$-weak limit of $(u_n)$. Denote by $f_n:=u_n^2$. First we show that $\|f\|_{L^1(\mathbb R^2)}=M$.  Up to an extraction, we may assume that a subsequence of $(f_n)$ (also denoted by $(f_n)$) converges weakly to $f$ in the sense  of distributions; that is for any $\varphi\in\mathcal C_0^\infty(\mathbb R^2)$
(smooth and compactly supported function), we have
$$
\int_{\mathbb R^2}\varphi f_n\;dx\to \int_{\mathbb R^2}\varphi f\;dx.
$$
To show strong convergence in $L^1$, we observe that (see for example \cite{Florescu})
$$
\limsup_n\|f_n-f\|_{L^1}\leq C(\{f_n,\;n=1,2,\cdot\cdot\}),
$$
where, for any subset $\mathcal A\subset L^1(\mathbb R^2)$, the function $C(\mathcal A)$ introduced by H. P. Rosenthal \cite{Rosenthal} is given by
$$
C(\mathcal A)=\inf_\varepsilon\sup_{|A|<\varepsilon}\sup_n\int_{A}f_n\;dx.
$$
Using H\"older inequality and the above bounds \eqref{minimiz seq}, we have for any $R>0$
\begin{eqnarray*}
\int_{A}f_n\;dx&\leq&\sqrt{|A|}\sqrt{\int_{A}f_n^2\;dx}+\frac1{R^2} \int_{A\cap\{|x|>R\}}|x|^2f_n\;dx\\
&\lesssim&\sqrt{\varepsilon}+\frac1{R^2} ,
\end{eqnarray*}
which clearly shows that $C(\{f_n,\;n=1,2,\cdot\cdot\})=0$, and thus $\|u_n-u\|_{L^2}\to0$ and $\|f\|_{L^1(\mathbb R^2)}=\|u\|_{L^2(\mathbb R^2)}^2=M$, as desired.
Moreover, by the lower semi-continuity of the norms, we have

$$
\frac12(\|x u\|_{L^2}^2+\frac12\|u\|_{L^4}^4)=\frac12(\||x|^2f\|_{L^1}+\frac12\|f^2\|_{L^2}^2)\leq \liminf_n\frac12(\|x u_n\|_{L^2}^2+\frac12\|u_n\|_{L^4}^4)\leq\nu_M.
$$
If the estimate  were  strict that would contradict the minimality of $\nu_M$. The convergence is therefore strong in $v_n\to v$,
and at the minimum we have

$$
|x|^2u+u^3=\nu u,\quad u^2=(\nu-|x|^2)_+
$$
yielding

$$
M=\|u_M\|_{L^2}^2=\int_{\mathbb R^2}(\nu-V)_+\;dx=\int_{\{|x|^2<\nu\}}(\nu-|x|^2)_+\;dx=\frac\pi2\nu^2,
$$
and

$$
\|u_M\|_{L^4}^4=\int_{\mathbb R^2}(\nu-|x|^2)_+|u|^2\;dx=\int_{\mathbb R^2}(\nu-|x|^2)_+^2\;dx\leq\frac\pi3\nu^3\sim M^\frac32.
$$
Now we mollify $v_M$ in order to get an upper bound for $\nu_M$. Set

$$
\tilde{u}_M:=\left((\nu-|x|^2)_+^2+1\right)^\frac14-1,\quad w_M:=\sqrt M\frac{\tilde{u}_M}{\|\tilde u_M\|_{L^2}}.
$$
Calculating $\|\tilde{u}_M\|_{L^2}^2$ shows that
\begin{eqnarray}
\label{key1}
\|\tilde{u}_M\|_{L^2}^2=\int_0^\mu\left((s^2+1)^\frac14-1\right)^2\;ds\sim\mu^2= M\quad\mbox{as}\quad M\to\infty.
\end{eqnarray}
Moreover, similar calculation enables us to see that

\begin{eqnarray}
\label{key2}
\|\nabla\tilde{u}_M\|_{L^2}^2\lesssim \nu^3\quad\mbox{and}\quad \||x|\tilde{u}_M\|_{L^2}^2\lesssim \nu^3.
\end{eqnarray}
In summary, in virtue of \eqref{key1} and \eqref{key2}, we have

\begin{eqnarray}
\label{key3}
\|w_M\|_{L^2}^2=M\quad\mbox{and}\quad \||x| w_M\|_{L^2}^2\lesssim M^\frac32,
\end{eqnarray}
which implies, thanks to the fact that  $\mathcal H(u_M)\leq \mathcal H(w_M)$,
$$
\|u_M\|_{L^2}^2=M,\quad\quad\mbox{and}\quad\quad\|xu_M\|_{L^2}^2\lesssim M^\frac32,\ as\ M\to\infty.
$$
The above estimates automatically imply

\begin{eqnarray}
\label{key}
M^\frac32\lesssim \|u_M\|_{L^4}^4.
\end{eqnarray}
Indeed, if \eqref{key}  does not hold, then there would exist a sequence $M_n\to\infty$, and $(u_n)_n$ satisfying
$$
\|u_n\|_{L^2}^2=M_n\quad\mbox{and}\quad \||x| u_n\|_{L^2}^2\lesssim M_n^\frac32
$$
and
$$
\|u_n\|_{L^4}^4\leq\frac{M_n^\frac32}n.
$$
On the other hand, for all $R>0$ and $n\in\mathbb N$
\begin{eqnarray*}
\|u_n\|_{L^2}^2&\lesssim&\frac{M_n^\frac32}{R^2}+R\|u_n\|_{L^4}^2\\
&\lesssim&\frac{M_n^\frac32}{R^2}+R\frac{M_n^\frac34}{n^\frac12}.
\end{eqnarray*}
Now choosing $R=M_n^\frac14n^\frac18$, gives the bound
$$
1\lesssim \frac1{n^\frac14}
$$
leading to a contradiction by taking $n\to\infty$.
Clearly, \eqref{key} shows that $\mathcal K(u_M)$ becomes negative as $M\to\infty$ which finishes the proof.
\end{proof}
Notice that to construct a nonlinear solution to \eqref{mu SP eps 0},
one can use several techniques. Variationnally, for any given amount of mass $M>0$, we have shown that a radial positive solution $(u_M,\mu_M)$ to \eqref{mu SP eps 0} can be constructed through the following minimizing problem
$$
\mu_M=\mathcal H(u_M):=\min _{\|u\|_{L^2}^2=M}\mathcal H(u).
$$
Moreover, this family of solutions is included in the branch of solutions   constructed using bifurcation arguments pioneered by Rabinowitz, and Crandall-Rabinowitz \cite{CR}.
Indeed, $(u,\mu)$ is a solution to \eqref{mu SP eps 0} if and only if $(I-\mu K)u=\mathcal N(u)$, where $K=A^{-1}B$,
$\mathcal N=A^{-1}G'(u)$, and the operators $A$, $B$ and $G $ are defined by

$$
A:\Sigma\to\Sigma^*,\quad \mbox{for any}\quad u,v \in\Sigma; \;<Au,v>:=(\nabla u,\nabla v)_2+(x u,x v)_2,
$$

$$
B:\Sigma\to\Sigma^*,\quad \mbox{for any}\quad u,v \in\Sigma; \;<Bu,v>:=(u,v)_2,
$$
and
$$
G:\Sigma\to\mathbb R,\quad \mbox{for any}\quad u \in\Sigma; \;G(u)=-\frac14\|u\|_{L^4}^4.
$$
Indeed, the following proposition shows that a branch of solutions of \eqref{mu SP eps 0} emerging from
the linear solution $(\varphi_1,\omega_1)$ can be constructed. The proof of the proposition is included in the proof of the spectral assumption given in the Appendix.  (See section 5). 
	\begin{proposition}
		\label{GS via bif}
		There exists $\eta_0>0$ such that for all $0<\eta<\eta_0$, a unique solution $u(\eta)\in\Sigma$, $\mu(\eta)>2$ of \eqref{mu SP eps 0}
		exists such that
		$$
		u(\eta)=\sqrt\eta (a(\eta)\varphi_1+ z(\eta)),
		$$
		with $z\in\Sigma$, $z(0)=0$ and $(z(\eta),\varphi_1)_2=0$.
	\end{proposition}
For the solution $(Q_0,\mu_0)$  to \eqref{mu SP eps 0} satisfying $\mathcal K(Q_0)=0$ given by Proposition \ref{exis Q0}, denote by
$$
L_-:=-\Delta+V+Q_0^{2}-\mu_0,
$$
and

$$
L_+:=-\Delta+V+3Q_0^{2}-\mu_0.
$$

The second preliminary result concerns the operators $L_\pm$. We have the following important property of $L_\pm$.

\begin{proposition}
	\label{kernal L-}
Let $<Q_0>^\bot$ be the subspace of $\Sigma$ consisting of all functions $L^2$-orthogonal to $Q_0$. Then we have
\begin{eqnarray*}
ker(L_-)=\{Q_0\}, \quad\mbox{and}\quad  L_-:<Q_0>^\bot\to<Q_0>^\bot\quad\mbox{is bijective}.
\end{eqnarray*}
Moreover, there exists $\alpha_0>0$ such that for all $\alpha>\alpha_0$,
\begin{eqnarray*}
L_+:\Sigma\to\Sigma^* \quad\mbox{is bijective}.
\end{eqnarray*}
\end{proposition}
The  property of $L_+$ comes from the breakdown of the spatial translation symmetry due to the presence of the potential. We refer to the Appendix (Section 5)  for the proof of
proposition \ref{kernal L-} .

We have
\begin{eqnarray}
\label{L Q 0}
 L_+(Q_0)=2Q_0^3.
\end{eqnarray}
Since $\mathcal K(Q_0)=(Q_0,(\sigma-\alpha|Q_0|^{2}) Q_0)_2=0$,
then thanks to Proposition \ref{kernal L-}, one can uniquely define $Q_{1i}$ by

$$
L_-Q_{1i}:=(\alpha|Q_0|^{2}-\sigma) Q_0.
$$
Observe that given the smoothness and the decay of $Q_0$, we have $Q_{1i}\in\text{Dom}L_-$. Moreover, we have
\begin{eqnarray}
\label{L Q i}
L_+^{-1}:L^2\to \mbox{Dom}(L_+)\quad \mbox{is bounded, and }\quad L_+(Q_{1i})=\alpha Q_0^3+2Q_0^2Q_{1i}-\sigma Q_0.
\end{eqnarray}
Now, define $Q_{2r}$ and $Q_{3i}$ by
\begin{equation}\label{qr}
L_+Q_{2r}=\mu_2Q_0+(\sigma-\alpha|Q_0|^{2}) Q_{1i}-Q_0Q_{1i}^2,
\end{equation}
and

\begin{equation}\label{qi3}
L_-Q_{3i}=(2Q_{2r}Q_0-Q_{1i}^2)Q_{1i}+\mu_2Q_{1i}+((2+\alpha)Q_0^2-\sigma)Q_{2r}+Q_{1i}^2Q_0.
\end{equation}
The bijectivity of $L_+$ enables us to determine $Q_{2r}$, and again the regularity of $Q_0$ shows that $Q_{2r}\in\text{Dom}L_+$. Thus it only remains to determine
the coefficient $\mu_2$, and $Q_{3i}$. They are  determined  by the orthogonality condition

$$
(L_-Q_{3i},Q_0)_2=0.
$$
Indeed, substituting $Q_{2r}$ (given by inverting \eqref{qr}) into \eqref{qi3} gives

\begin{eqnarray}
L_-Q_{3i}&=&\mu_2[ Q_{1i}+ ((2+\alpha)Q_0^2-\sigma+2Q_0Q_{1i})  L_+^{-1}Q_0 ]+
Q_{1i}^2Q_0-Q_{1i}^3\\
&+&((2+\alpha)Q_0^2-\sigma+2Q_0Q_{1i})L_+^{-1}\big((\sigma-Q_0^2)Q_{1i}-Q_0Q_{1i}^2\big).
\end{eqnarray}
Now since
$(Q_0,Q_{1i})_2=0$, then clearly

$$
(L_+^{-1}((2+\alpha)Q_0^2-\sigma+2Q_0Q_{1i}),Q_0)=\|Q_0\|_{L^2}^2\neq0,
$$
which insures that $\mu_2$ is uniquely determined in terms of $Q_0$, $Q_{1,i}$ which were already defined. Then $Q_{3i}$ follows by inverting $L_-$ using the orthogonality
$(Q_{3i},Q_0)_2=0$. Now, set

\begin{eqnarray}
 \label{prox sol}
Q_\varepsilon^a:=Q_0+i\varepsilon Q_{1i}+\varepsilon^2 Q_{2r}+i\varepsilon^3 Q_{3i},
\quad\mbox{and}\quad \mu_\varepsilon^a=\mu_0+\varepsilon^2\mu_2.
\end{eqnarray}
The main result of this section is the following.
\begin{theorem}\label{ground state}
For $\sigma$ a heaviside function and $\alpha>\alpha_0$, there exists $\varepsilon_0>0$ such that for all $0<\varepsilon<\varepsilon_0$, equation \eqref{NLS eps} 
has a solution
$(Q_\varepsilon,\mu_\varepsilon)\in \Sigma\times (2,\infty)$
that can be decomposed as
\begin{eqnarray}
(Q_\varepsilon,\mu_\varepsilon)=(Q_\varepsilon^a+\psi_\varepsilon,\mu_\varepsilon^a+\kappa_\varepsilon),
\end{eqnarray}
with $\psi_\varepsilon=\psi_{\varepsilon, r}+i\psi_{\varepsilon,i}$ satisfying
\begin{eqnarray}
 |\kappa_\varepsilon|+\|\psi_{\varepsilon,r}\|_{\Sigma}&\lesssim& \varepsilon^4\\
\|\psi_{\varepsilon,i}\|_{\Sigma}&\lesssim& \varepsilon^5.
\end{eqnarray}

\end{theorem}

\begin{proof}[Proof of Theorem \ref{ground state}]
First, we write an equation for $(Q_\varepsilon,\mu_\varepsilon)$ being a solution of  \eqref{mu SP eps}. We start by further decomposing $Q_\varepsilon^a=Q_{\varepsilon, r}^a+iQ_{\varepsilon,i}^a$ and observe that
$$
|Q_\varepsilon|^2=|Q_{\varepsilon, r}^a|^2+|Q_{\varepsilon, i}^a|^2
+2Q_{\varepsilon, r}^a\psi_{\varepsilon, r}+2Q_{\varepsilon, i}^a\psi_{\varepsilon, i}
+|\psi_{\varepsilon, r}|^2+|\psi_{\varepsilon, i}|^2.
$$
Substituting this in equation \eqref{mu SP eps} and splitting the real and imaginary parts, we obtain
\begin{eqnarray}\nonumber
(\mu_\varepsilon^a+\kappa_\varepsilon)(Q_{\varepsilon, r}^a+\psi_{\varepsilon, r})&=
(-\Delta+V+|Q_\varepsilon|^2)(Q_{\varepsilon, r}^a+\psi_{\varepsilon, r})\\&-
\varepsilon(\sigma-\alpha|Q_\varepsilon|^2)(Q_{\varepsilon, i}^a+\psi_{\varepsilon, i}),
\label{real part}
\end{eqnarray}
and
\begin{eqnarray}\nonumber
 (\mu_\varepsilon^a+\kappa_\varepsilon)(Q_{\varepsilon, i}^a+\psi_{\varepsilon, i})&=
(-\Delta+V+|Q_\varepsilon|^2)(Q_{\varepsilon, i}^a+\psi_{\varepsilon, i})\\&+
\varepsilon(\sigma-\alpha|Q_\varepsilon|^2)(Q_{\varepsilon, r}^a+\psi_{\varepsilon, r}),
\label{imag part}
\end{eqnarray}
respectively.
The identity coming from the real part can be rewritten in the following way.
\begin{eqnarray*}\nonumber
L_+\psi_{\varepsilon, r}&=&\mu_{\varepsilon}^aQ_{\varepsilon, r}^a-(-\Delta+V+|Q_\varepsilon^a|^2)\psi_{\varepsilon, r}+
\varepsilon(\sigma-\alpha|Q_\varepsilon^a|^2)Q_{\varepsilon, i}^a\\
&+&\kappa_{\varepsilon}Q_{\varepsilon, r}^a+\varepsilon^2\mu_2\psi_{\varepsilon, r}-
2Q_{\varepsilon, i}^aQ_{\varepsilon, r}^a\psi_{\varepsilon, r}+  \varepsilon(\sigma-\alpha|Q_\varepsilon^a|^2)\psi_{\varepsilon, i}^a \\
&-&2|Q_{\varepsilon, r}^i|^2\psi_{\varepsilon, i}^a+ \kappa_\varepsilon\psi_{\varepsilon, r}+
\psi_{\varepsilon, r}(2Q_{\varepsilon, r}^a\psi_{\varepsilon, r}+2Q_{\varepsilon, i}^a\psi_{\varepsilon, i}+\psi_{\varepsilon, r}^2+\psi_{\varepsilon, i}^2)\\
&-&\varepsilon\psi_{\varepsilon, r}(2Q_{\varepsilon, r}^a\psi_{\varepsilon, r}+2Q_{\varepsilon, i}^a\psi_{\varepsilon, i}+\psi_{\varepsilon, r}^2+\psi_{\varepsilon, i}^2)\\
&:=&\kappa_\varepsilon Q_0+\varepsilon^4g_1+F_\varepsilon(\psi_{\varepsilon, r},\psi_{\varepsilon, i},\kappa_\varepsilon)
\end{eqnarray*}
where $g_1$ is given by
$$
g_1:= \mu_2Q_{2r}-Q_{1i}^2Q_{2r}-(Q_{2r}^2+2Q_{3i}Q_{1i})Q_{0}+(\sigma-\alpha Q_0^2)Q_{3i}-(2Q_0Q_{2r}+Q_{1i}^2)Q_{1i}
$$
and $F_\varepsilon$ can be  explicitly computed. In particular it satisfies
$$
\|F_\varepsilon(\psi_{\varepsilon, r},\psi_{\varepsilon, i},\kappa_\varepsilon)\|_{\Sigma}\lesssim \varepsilon^6.
$$
The identity coming from the imaginary part can be rewritten in the following way.
\begin{eqnarray*}\nonumber
L_-\psi_{\varepsilon, i}&=&\mu_{\varepsilon}^aQ_{\varepsilon, i}^a-(-\Delta+V+|Q_\varepsilon^a|^2)Q_{\varepsilon, i}^a-
\varepsilon(\sigma-\alpha|Q_\varepsilon^a|^2)Q_{\varepsilon, r}^a\\
&+&\kappa_{\varepsilon}Q_{\varepsilon, i}^a+\varepsilon^2\mu_2\psi_{\varepsilon, i}-
2Q_{\varepsilon, i}^a(Q_{\varepsilon, r}^a\psi_{\varepsilon, r}+Q_{\varepsilon, i}^a\psi_{\varepsilon, i})
-\varepsilon(\sigma-\alpha|Q_0|^2)\psi_{\varepsilon, r}\\
&+&
2\varepsilon Q_{\varepsilon, r}^a( Q_{\varepsilon, r}^a \psi_{\varepsilon, r}+ Q_{\varepsilon, i}^a\psi_{\varepsilon, i} )\\
&-&2\psi_{\varepsilon, i}(Q_{\varepsilon, r}^a \psi_{\varepsilon, r}+ Q_{\varepsilon, i}^a\psi_{\varepsilon, i} )+
2\varepsilon \psi_{\varepsilon, r}(Q_{\varepsilon, r}^a\psi_{\varepsilon, r}+Q_{\varepsilon, i}^a\psi_{\varepsilon, i})\\
&+&\varepsilon Q_{\varepsilon, r}^a(\psi_{\varepsilon, r}^2+\psi_{\varepsilon, i}^2)-
\psi_{\varepsilon, i}(\psi_{\varepsilon, r}^2+\psi_{\varepsilon, i}^2)
+\varepsilon \psi_{\varepsilon, r}(\psi_{\varepsilon, r}^2+\psi_{\varepsilon, i}^2)+\kappa_\varepsilon \psi_{\varepsilon,i}
\label{real part2}\\
&:=&\varepsilon\big(\kappa_\varepsilon Q_{1i}+((2+\alpha)Q_0^2-\sigma-2Q_0Q_{1i})\psi_{\varepsilon,r}\big)\\
&+&\varepsilon^5\varphi_2+G_\varepsilon(Q_{\varepsilon, r},Q_{\varepsilon, i},\kappa_\varepsilon),
\end{eqnarray*}
where $\varphi_2$ is given by
$$
\varphi_2:=-(2Q_0Q_{2r}+Q_{1i}^2) Q_{3i} -(Q_{2r}^2+2Q_{1i}Q_{3i})Q_{1i} +(2Q_{0}Q_{2r}+Q_{1i}^2) Q_{2r}
+(Q_{2r}^2+2Q_{1i}Q_{3i})Q_0
$$
and $G_\varepsilon$ can be explicitely computed. In particular it satisfies
$$
\|G_\varepsilon(\psi_{\varepsilon, r},\psi_{\varepsilon, i},\kappa_\varepsilon)\|_{\Sigma}\lesssim \varepsilon^7.
$$
Now we define a map $\Phi_\varepsilon:\Sigma\times \Sigma\times(0,\infty)\to\Sigma\times \Sigma\times(0,\infty)$ by
$$
\Phi_\varepsilon(\tilde\psi_{\varepsilon, r},\tilde\psi_{\varepsilon, i},\tilde\kappa_\varepsilon)=
(\psi_{\varepsilon, r},\psi_{\varepsilon, i},\kappa_\varepsilon)
$$
where, $(\psi_{\varepsilon, r},\psi_{\varepsilon, i},\kappa_\varepsilon)$ solves
\begin{equation}
\label{L+ contrac} \left\{
 \begin{aligned}
  L_+\psi_{\varepsilon, r}&=\kappa_\varepsilon Q_0+\varepsilon^4g_1+
F_\varepsilon(\tilde\psi_{\varepsilon, r},\tilde\psi_{\varepsilon, i},\tilde\kappa_\varepsilon) \\
L_-\psi_{\varepsilon, i}&=\varepsilon\big(\kappa_\varepsilon Q_{1i}+((2+\alpha)Q_0^2-\sigma-2Q_0Q_{1i})\psi_{\varepsilon,r}\big)
+\varepsilon^5\varphi_2+G_\varepsilon(\tilde\psi_{\varepsilon, r},\tilde\psi_{\varepsilon, i},\tilde\kappa_\varepsilon),\\
(L_-\psi_{\varepsilon, i},Q_0)_2&=0.
 \end{aligned}
\right.
\end{equation}
Now the purpose is to show that there are positive constants $C_1,C_2$ and $C_3$ such that the above map is a
contraction on the ball
$$
B_\varepsilon:=\{(\psi_{\varepsilon, r},\psi_{\varepsilon, i},\kappa_\varepsilon):\quad
|\kappa_\varepsilon|\leq C_1 \varepsilon^4,\;\|\psi_{\varepsilon, r}\|_{\Sigma}\leq C_2\varepsilon^4,\;
\|\psi_{\varepsilon, i}\|_{\Sigma}\leq C_3\varepsilon^5\},
$$
for $\varepsilon>0$ sufficiently small. The ball $B_\varepsilon$ is endowed with the norm
\begin{eqnarray}
 \label{norm}
\max\left\{\frac{|\kappa_\varepsilon|}{ C_1 \varepsilon^4},   \;\frac{\|\psi_{\varepsilon, r}\|_{\Sigma}}{C_2\varepsilon^4},\;
\frac{\|\psi_{\varepsilon, i}\|_{\Sigma}}{C_3\varepsilon^5}      \right\}.
\end{eqnarray}
\\
Thanks to the equation on $\psi_{\varepsilon,r}$ and the invertibility of $L_+$, we can write
\begin{eqnarray}\label{psi eps r}
\psi_{\varepsilon,r}=\kappa_\varepsilon Q'(\mu_0)+\varepsilon^4L_+^{-1}(g_1)+L_+^{-1}
\big(F_\varepsilon(\tilde\psi_{\varepsilon, r},\tilde\psi_{\varepsilon, i},\tilde\kappa_\varepsilon)\big).
\end{eqnarray}
Plugging the above identity in the equation on $\psi_{\varepsilon,i}$, we obtain
\begin{eqnarray}\nonumber
L_-\psi_{\varepsilon, i}&=\varepsilon\kappa_\varepsilon\big( Q_{1, i}+((2+\alpha)Q_0^2+2Q_0Q_{1, i}-\sigma)Q'(\mu_0)\\ \nonumber
&+\varepsilon^5((2+\alpha)Q_0^2+2Q_0Q_{1, i}-\sigma)L_+^{-1}(g_1)+L_+^{-1}
\big(F_\varepsilon(\tilde\psi_{\varepsilon, r},\tilde\psi_{\varepsilon, i},\tilde\kappa_\varepsilon)\big).
\end{eqnarray}
Since,
$$
\big(Q_{1, i}+((2+\alpha)Q_0^2-2Q_0Q_{1, i}-\sigma)Q'(\mu_0),Q_0\big)_2=\|Q_0\|_{L^2}^2
$$
then the choice of
$$
\kappa_\varepsilon=-\frac{\varepsilon^4}{\|Q_0\|_{L^2}^2}\int_{\mathbb R^2}Q_0g_1\;dx-
\frac1\varepsilon\int_{\mathbb R^2}Q_0L_+^{-1}
\big(F_\varepsilon(\tilde\psi_{\varepsilon, r},\tilde\psi_{\varepsilon, i},\tilde\kappa_\varepsilon)\big)\;dx
$$
makes
$$
(L_-\psi_{\varepsilon, i},Q_0)_2=0
$$
,which enables us to invert $L_-$ and thus calculate $\psi_{\varepsilon, i}$:

\begin{eqnarray}\nonumber
 \psi_{\varepsilon, i}&=&\varepsilon\kappa_\varepsilon L_-^{-1}\big( Q_{1, i}+((2+\alpha)Q_0^2-2Q_0Q_{1, i}-\sigma)Q'(\mu_0)\big)+
L_-^{-1}L_+^{-1}
\big(F_\varepsilon(\tilde\psi_{\varepsilon, r},\tilde\psi_{\varepsilon, i},\tilde\kappa_\varepsilon)\big)
\\
&+&
\varepsilon^5L_-^{-1}\big(((2+\alpha)Q_0^2-2Q_0Q_{1, i}-\sigma)L_+^{-1}(g_1)\big)
\end{eqnarray}
Let
\begin{eqnarray}
 \label{constants c1c2c3}
C_1:=\frac{2\|g_1\|_{L^2}}{\|Q_0\|_{L^2}},
\nonumber
\end{eqnarray}
\begin{eqnarray}
C_2=2\big(C_1\|Q'(\mu_0)\|_\Sigma+\|L_+^{-1}(g_1)\|_\Sigma\big),\nonumber
 \end{eqnarray}
and
\begin{eqnarray}
C_3:&=&2C_1\|L_-^{-1}\big( Q_{1, i}+((2+\alpha)Q_0^2-2Q_0Q_{1, i}-\sigma)Q'(\mu_0)\big)\|_\Sigma\nonumber\\
&+&
2\|L_-^{-1}\big(((2+\alpha)Q_0^2-2Q_0Q_{1, i}-\sigma)L_+^{-1}(g_1)\big)\|_\Sigma\nonumber
\end{eqnarray}
To show that $\Phi_\varepsilon$ is a contraction, consider
$(\tilde\psi_{\varepsilon, r}^a,\tilde\psi_{\varepsilon, i}^a,\tilde\kappa_\varepsilon^a)$ and
$(\tilde\psi_{\varepsilon, r}^b,\tilde\psi_{\varepsilon, i}^b,\tilde\kappa_\varepsilon^b)$ in the ball $B_\varepsilon$ and denote by
$(\psi_{\varepsilon, r}^a,\psi_{\varepsilon, i}^a,\kappa_\varepsilon^a)$ and
$(\psi_{\varepsilon, r}^b,\psi_{\varepsilon, i}^b,\kappa_\varepsilon^b)$ their respective images through the map $\Phi_\varepsilon$. We have

$$
\kappa_\varepsilon:=\kappa_\varepsilon^a-\kappa_\varepsilon^b=
\frac1\varepsilon\int_{\mathbb R^2}Q_0L_+^{-1}
\big[F_\varepsilon(\tilde\psi_{\varepsilon, r}^b,\tilde\psi_{\varepsilon, i}^b,\tilde\kappa_\varepsilon^b)-
F_\varepsilon(\tilde\psi_{\varepsilon, r}^a,\tilde\psi_{\varepsilon, i}^a,\tilde\kappa_\varepsilon^a)
\big]\;dx,
$$

$$
\psi_{\varepsilon, r}:=\psi_{\varepsilon, r}^a-\psi_{\varepsilon, r}^b=\kappa_\varepsilon Q'(\mu_0)-L_+^{-1}
\big(F_\varepsilon(\tilde\psi_{\varepsilon, r}^b,\tilde\psi_{\varepsilon, i}^b,\tilde\kappa_\varepsilon^b)-
F_\varepsilon(\tilde\psi_{\varepsilon, r}^a,\tilde\psi_{\varepsilon, i}^a,\tilde\kappa_\varepsilon^a)
\big),
$$
and
\begin{eqnarray}\nonumber
\psi_{\varepsilon, i}:=\psi_{\varepsilon, r}^a-\psi_{\varepsilon, r}^b&=&
\varepsilon\kappa_\varepsilon L_-^{-1}\big( Q_{1, i}+((2+\alpha)Q_0^2-2Q_0Q_{1, i}-\sigma)Q'(\mu_0)\big)\\&-&
L_-^{-1}L_+^{-1}
\big(F_\varepsilon(\tilde\psi_{\varepsilon, r}^b,\tilde\psi_{\varepsilon, i}^b,\tilde\kappa_\varepsilon^b)-
F_\varepsilon(\tilde\psi_{\varepsilon, r}^a,\tilde\psi_{\varepsilon, i}^a,\tilde\kappa_\varepsilon^a)
\big).
\nonumber
\end{eqnarray}
Estimating $\kappa_\varepsilon$, $\psi_{\varepsilon, r}$ and $\psi_{\varepsilon, i}$ using the above bounds on $F_\varepsilon$ and $G_\varepsilon$ yields

\begin{eqnarray}\nonumber
|\kappa_\varepsilon|\lesssim \varepsilon^5, \quad \|\psi_{\varepsilon, r}\|_\Sigma\lesssim\varepsilon^5,\quad  \|\psi_{\varepsilon, i}\|_\Sigma\lesssim\varepsilon^6
\nonumber
\end{eqnarray}
showing the contraction of the map $\Phi_\varepsilon$
\end{proof}

\section{The Cauchy Problem}
In this section, we study the Cauchy problem:
\begin{equation}\label{eq11}\left\{ \begin{array}{ll}
i\partial_t\psi + \Delta\psi &= V(x) \psi + |\psi|^2 \psi +
i(\sigma(x) - \alpha |\psi|^2) \psi,\quad t > 0, x \in
\mathbb{R}^2,\\
\psi_{|t=0}& = \psi_0.
\end{array}\right.
\end{equation}
We will first assume that $\sigma\in L^4(\mathbb{R}^2)$, we set $U(t) = e^{-it(-\Delta +V)}$.\\
\begin{definition}\label{def 4.1}
A pair $(p,q)$ is admissible if $2 \leq q <
\infty$ and
$$\frac{2}{p} = \delta(q) : = 2 (\frac{1}{2} - \frac{1}{q}).$$
\end{definition}
Recall the following Strichartz estimates for the Schr\"odinger equation with potential  are due to \cite{carles}.
\begin{proposition}\label{prop 4.2}
Let $T> 0$.
\begin{enumerate}
\item For any admissible pair $(p,q)$, there exists $C_q(T)$ such
that
$$\|U(.)\varphi\|_{L^p([0,T] ;L^q)} \leq C_q(T)\|\varphi\|_{L^2},\quad \forall\; \varphi \in L^2(\mathbb{R}^2).\eqno{(4.2)}$$
\item Denote
$$D(F)(t,x) = \int^t_0 U(t - s) F(s,x)d\tau.$$
For all admissible pairs $(p_1,q_1)$ and $(p_2,q_2)$, there exists
$C = C_{q_1,q_2}(T)$ such that
$$\|D(F)\|_{L^{p_1}([0,\tau]; L^{q_1})} \leq C\|F\|_{{L^{p'_2}}([0,\tau]; L^{q'_2})},\eqno{(4.3)}$$
for all $F \in L^{p'_2}([0,T] ; L^{q'_2})$ and $0 \leq \tau \leq
T$.
\end{enumerate}
\end{proposition}

\begin{proposition}
\label{prop 4.3}
There exists $\delta > 0$ such that if
$\psi_0 \in L^2(\mathbb{R}^2)$ and $T \in [0,1]$ are  such that
$$\|U(.)\psi_0\|_{L^4([0,T]\times \mathbb{R}^2)} \leq \delta \quad
\mbox{ and } \quad T^{3/4}\|\sigma\|_{L^4(\mathbb{R}^2)}\leq
\frac{1}{8}\;,$$ then (4.1) has a unique solution
$$\psi \in C([0,T] ; L^2) \cap L^4 ([0,T] \times \mathbb{R}^2).$$
\end{proposition}

\begin{proof}[Proof of Proposition \ref{prop 4.3}] Let
$$X = \{\psi \in C([0,T] ; L^2) \cap L^4([0,T] \times
\mathbb{R}^2),\quad \|\psi\|_{L^4([0,T] \times \mathbb{R}^2)} \leq
2\delta\}.$$
In view of Duhamel's formula, and for $\psi\in X$, introduce the map
$$\Phi(\psi)(t) : = U(t)\psi_0 - i \int^t_0 U(t-s)(1-i\alpha)|\psi|^2 \psi(s)ds +
 \int^t_0 U(t-s)(\sigma\psi)(s)ds.
 $$
 From Strichartz inequalities
 \begin{eqnarray*}
 \|\Phi(\psi)\|_{L^4([0,T]\times \mathbb{R}^2)} &\leq&
 \|U(.)\psi_0\|_{L^4([0,T]\times \mathbb{R}^2)} +
 C(1+\alpha)\||\psi|^2\psi\|_{L^{4/3}([0,T]\times \mathbb{R}^2)}\\
 &+& \|\sigma \psi\|_{L^1([0,T]; L^2)}\\
 &\leq& \delta + C(1+\alpha\|\psi\|^3_{L^4([0,T]\times
 \mathbb{R}^2)} + \|\sigma\|_{L^4}\|\psi\|_{L^1([0,T];L^4)}\\
 &\leq& \delta + C(1+\alpha)(2\delta)^3 +
 \|\sigma\|_{L^4}T^{3/4}\|\psi\|_{L^4([0,T] \times \mathbb{R}^2)}\\
 &\leq& \delta + C(1+\alpha)(2\delta)^3 + 2\delta
 T^{3/4}\|\sigma\|_{L^4}\\
 &\leq& \delta + C(1+\alpha)(2\delta)^3 + \frac{\delta}{4}.
 \end{eqnarray*}
By
 choosing $\delta > 0$ sufficiently small, the right hand side does
 not exceed $2\delta : X$ is stable under the action of $\Phi$. For
 the contraction, let $\psi_1, \psi_2 \in X $ :
 \begin{eqnarray*}
 \|\Phi(\psi_1) - \Phi(\psi_2)\|_{L^4([0,T]\times \mathbb{R}^2)}
 &\leq& C(1+\alpha)\||\phi_1|^2\psi_1 - |\psi_2|^2 \psi_2\|_{L^{4/3}([0,T]\times
 \mathbb{R}^2)}\\
 &+& \|\sigma(\psi_1-\psi_2)\|_{L^1([0,T];L^2)}\\
 &\leq& C(\|\psi_1\|^2_{L^4([0,T]\times \mathbb{R}^2)} +
 \|\psi_2\|^2_{L^4([0,T]\times \mathbb{R}^2)})\|\psi_1-\psi_2\|_{L^4([0,T]\times
 \mathbb{R}^2)}\\
 &+& \|\sigma\|_{L^4}T^{3/4}\|\psi_1 - \psi_2\|_{L^4([0,T] \times
 \mathbb{R}^2)}\\
 &\leq& (C\delta^2 + \frac{1}{8})\|\psi_1 - \psi_2\|_{L^4([0,T]\times
 \mathbb{R}^2)}.
 \end{eqnarray*}
 Up to decreasing $\delta$ again, the factor on the right hand side
 does not exceed $1/2$ and $\Phi$ is a contraction on $X$. This proves the
 existence part of the proposition. The uniqueness part readily
 follows from the remark that if $\psi \in L^4([0,T] \times
 \mathbb{R}^2)$, then $[0,T]$ can be split finitely many times on
 intervals where
 $$\|\psi\|_{L^4(I_j \times \mathbb{R}^2)} \leq 2\delta,$$
 so uniqueness on $X$ can be deduced.
 \end{proof}

 \begin{theorem}\label{th4.1}Let $\psi_0 \in L^2(\mathbb{R}^2), \sigma \in
 L^4(\mathbb{R}^2)$. Then (4.1) has a unique, maximal solution
 $$\psi \in C([0, T_{max}) ; L^2)\cap L^4_{loc}([0, T_{max}) ; L^4(\mathbb{R}^2)).$$
Moreover,  in $[0, T_{max})$:
 $$\frac{d}{dt} \|\psi(t)\|^2_{L^2} + \alpha\|\psi(t)\|^4_{L^4} -
\int_{\mathbb{R}^2}\sigma(x)|\psi(t,x)|^2 dx = 0.\eqno{(4.4)}$$ It
is maximal in the sense that if $T_{max}$ is finite, then
$$\int^{T_{max}}_0\int_{\mathbb{R}^2}|\psi(t,x)|^4 dtdx = \infty.$$
\end{theorem}

\begin{proof} Since $\psi_0 \in L^2$, the homogeneous Strichartz
inequality (2.1) implies $U(.) \psi_0 \in L^4([0,1]) \times
\mathbb{R}^2)$, hence
$$\|U(.)\psi_0\|_{L^4([0,T]\times \mathbb{R}^2)}
\displaystyle{\mathop{\mapsto}_{T\rightarrow 0}}0.$$
Moreover, Proposition \ref{prop 4.3} yields a local solution satisfying \eqref{eq11}. 
For the notion of maximality, we proceed as in [2]. Suppose that
$\psi \in C([0, T_{max}) ; L^2) \cap L^4([0, T_{max}] \times
\mathbb{R}^2)$, with $T_{max}$ finite, and that $\psi$ cannot be
extended to larger time. Let $t \in [0, T_{max})$ and $s \in [0,
T_{max}-t)$. Duhamel's formula implies
$$U(s)\psi(t) = \psi(t+s)+i \int^s_0U(s-s')(1-i\alpha)|\psi|^2\psi(t+s')ds' - \int^s_0
U(s-s')(\sigma\psi) (t+s')ds'.$$ In view of the same inequality as
in the proof of Proposition 5.1.
\begin{eqnarray*}
\|U(.)\psi(t)\|_{L^4((0, T_{max}-t)\times \mathbb{R}^2)} &\leq&
\|\psi\|_{L^4((t,T_{max}) \times \mathbb{R}^2)} +
C\|\psi\|^3_{L^4(t,T_{max}) \times \mathbb{R}^2}\\
&&+ \frac{1}{8}\|\psi\|_{L^4((t,T_{max})\times \mathbb{R}^2)}.
\end{eqnarray*}
The right hand side is less than $\delta$ if $t$ is close to
$T_{max}$. Proposition 4.3 shows that $\psi$ can be extended  after
$T_{max}$, in contradiction with the definition of $T_{max}$.
\end{proof}

\begin{corollary}\label{4.5} If $\psi_0 \in L^2(\mathbb{R}^2)$ and $\sigma
\in L^4 \cap L^\infty(\mathbb{R}^2)$, then $T_{max} = \infty$, and
for all $t \geq 0$,
$$\|\psi(t)\|^2_{L^2} \leq \|\psi_0\|^2_{L^2}e^{t\|\sigma\|_{L^\infty}}.$$
\end{corollary}
\begin{proof} Form (5.1)
$$\frac{d}{dt} \|\psi(t)\|^2_{L^2} + \alpha\|\psi(t)\|^4_{L^4} - \|\sigma\|_{L^\infty}
\|\psi(t)\|^2_{L^2} \leq 0,$$ hence
$$\frac{d}{dt} \left(e^{-\|\sigma\|_{L^\infty}t} \|\psi(t)\|^2_{L^2}\right)
+ \alpha e^{-\|\sigma\|_{L^\infty}t}\|\psi(t)\|^4_{L^4} \leq 0,$$
and
$$e^{-\|\sigma\|_{L^\infty}t}\|\psi(t)\|^2_{L^2} + \alpha \int^T_0 e^{-\|\sigma\|_{L^\infty}t}
\|\psi(t)\|^4_{L^4}dt \leq \|\psi_0\|^2_{L^2}.$$ Therefore, for all
$T$ finite,
$$\int^T_0 \|\psi(t)\|^4_{L^4}dt \leq \frac{e^{T\|\sigma\|_{L^\infty}}}{\alpha}\;
\|\psi_0\|^2_{L^2}\;,$$ hence $T_{max} = \infty$ in Theorem \ref{th4.1}.
\end{proof}

\begin{corollary}\label{4.6} If $\psi_0 \in \Sigma$ and $\sigma \in L^4 \cap
W^{1.\infty}(\mathbb{R}^2)$, then (4.1) has a unique, global
solution $\psi$, such that
$$\psi, \nabla \psi , x \psi \in C([0, \infty) ; L^2(\mathbb{R}^2)) \cap L^4_{loc}
([0, \infty) ; L^4(\mathbb{R}^2)).$$
\end{corollary}

The analogue of Proposition \ref{prop 4.3} becomes, if we just assume $\sigma
\in L^\infty(\mathbb{R}^2)$ :\\
{\bf Proposition 4.7.} There exists $\delta > 0$ such that if
$\psi_0 \in L^2(\mathbb{R}^2)$ and $T \in (0,1]$ are such that
$$\|U(.)\psi_0\|_{L^4([0,T]\times \mathbb{R}^2)} \leq \delta \quad \mbox{ and }\quad
T\|\sigma\|_{L^\infty}\|\psi_0\|_{L^2} \leq \frac{\delta}{8}\;,$$
then (4.1) has a unique solution
$$\psi \in C([0,T] ; L^2) \cap L^4([0,T] \times \mathbb{R}^2).$$

The proof is  similar to the proof of Proposition \ref{prop 4.3}, by
working in
$$Y = \{\psi \in C([0,T] ; L^2)\cap L^4([0,T] \times \mathbb{R}^2), \;
\|\psi\|_{L^4([0,T] \times \mathbb{R}^2)} \leq 2\delta,
\|\psi\|_{L^\infty([0,T];L^2)} \leq 2\|\psi_0\|_{L^2}\}.$$ and
estimating
$$\|\sigma \psi\|_{L^1([0,T];L^2)} \leq T\|\sigma\|_{L^\infty([0,T];L^2)}.$$
Now, we still have (5.1), hence
$$\|\psi(t)\|^2_{L^2} \leq \|\psi_0\|^2_{L^2}
e^{t\|\sigma\|_{L^\infty}}\quad \mbox{ and } \quad \int^T_0
\|\psi(t)\|^4_{L^4}dt \leq \frac{e^T\|\sigma\|_{L^\infty}}{\alpha}\;
\|\psi_0\|^2_{L^2} .$$ Therefore, the solution is global again.


\section{Appendix: Proof of Proposition \ref{kernal L-}}

\begin{proof}
We start by proving (i). Consider the minimizing problem
	
	$$
	\ell_{\mu_0}:=\inf\{<L_-v;v>_{\Sigma^*,\Sigma}:\; u\in\Sigma\quad\mbox{and}\quad \|v\|_{L^2}=1\},
	$$
and observe that since $L_-(Q_0)=0$, then $\ell_{\mu_0}\leq0$.\\
On the one hand, arguing as in the proof of proposition \ref{exis Q0}, one can easily show that a minimizer $u$ of the above problem exists. Next, for any test function $\varphi$, we have
	$$
	<L_-(u+\varepsilon\varphi);(u+\varepsilon\varphi)>_{\Sigma^*,\Sigma} \geq \ell_{\mu_0}(1+\varepsilon^2\|\varphi\|_{L^2}^2+2\varepsilon Re(u,\varphi)_2).
	$$
That is
	
	\begin{eqnarray*}
		\ell_{\mu_0}&+&\varepsilon^2\left(  \|\nabla\varphi\|_{L^2}^2+\|x\varphi\|_{L^2}^2+\|Q_0\varphi\|_{L^2}^2-\mu_0\|\varphi\|_{L^2}^2   \right)\\
		&+& 2\varepsilon Re\left( (\nabla u,\nabla\varphi)_2+(x u,x\varphi)_2+(Q_0 u,Q_0\varphi)_2-
		\mu_0 (u,\varphi)_2   \right)\\
		&\geq&  \ell_{\mu_0}(1+\varepsilon^2\|\varphi\|_{L^2}^2+2\varepsilon Re(u,\varphi)_2).
	\end{eqnarray*}
Since $\varepsilon$ is arbitrary and can have any sign, then we deduce that
	$$
	L_-u=\ell_{\mu_0}u,
	$$
and therefore $u$ is an eigenvector of $L_-$ corresponding to the first  (and simple) eigenvalue $\ell_{\mu_0}$.
On the other hand, it follows from Theorem 11.8 in \cite{Lieb-Loss}
	that the minimizer is unique and up to a phase change, we can take a positive minimizer $\tilde u=e^{i\theta}u$ with $\theta\in\mathbb R$. Now, to conclude it suffices to show that $\ell_{\mu_0}=0$. We see from $L_-Q_0=0$ that
	$$
	\ell_{\mu_0}(\tilde u,Q_0)_2=<L_-\tilde u;Q_0>_{\Sigma^*,\Sigma}=<\tilde u;L_-Q_0>_{\Sigma^*,\Sigma}=0,
	$$
yielding (given that $Q_0$ and $\tilde u$ are positive) $\ell_{\mu_0}=0$. This finishes the proof.\\
Now we prove (ii). The proof of the bijectivity goes through several steps. First we prove Proposition \ref{GS via bif}.  Set

\begin{eqnarray}
\nonumber  \mu&=& 2+\eta \\
\label{5.0}  u &=& \sqrt{\eta} q, \mbox{ where } 0< \eta<<1.
\end{eqnarray}

Then $(\mu_0\!\!-\!\!{\rm SP}_0)$ reads as

$$(-\Delta + |x|^2 -2) q = \eta(q-q^3).$$
Now we can decompose $q= a\varphi_1 + \varphi_1^\perp$, where $\varphi_1$ is the first simple eigenfunction of the operator  $H_0:= -\Delta+|x|^2$,   $ \varphi_1^\perp$ denotes an element of the vector space $L^2$-orthogonal  to $\varphi_1$, and $a(\eta)$ is a scalar.
Therefore, we have

\begin{eqnarray}
\nonumber L_2\varphi_1^\perp:=(-\Delta + |x|^2 -2) \varphi_1^\perp&=& \eta\Big[a \varphi_1+\varphi^\perp -(a \varphi_1+ \varphi_1^\perp)^3\Big] \\
\label{5.1}   &=& \eta F(a, \varphi_1^\perp).
\end{eqnarray}
Now let us  define the projection $\Pi$ by: $\Pi(c\varphi_1+ \varphi_1^\perp)= \varphi_1^\perp.$ Thus, \eqref{5.1} can be rewritten in the following way

\begin{equation}\label{5.2}
    \varphi_1^\perp = \eta (\Pi L_2)^{-1} \Pi F(a, \varphi_1^\perp)
\end{equation}
and
\begin{equation}\label{5.3}
    (I-\Pi) \: F(a, \varphi_1^\perp)= 0.
\end{equation}
\noindent We will first solve \eqref{5.2} using the implicit function theorem.\\
First, we notice that for $\eta=0$, and $a_0$ satisfying:
$$
|\varphi_1|_2^2 -a_0^2 |\varphi_1|_4^4=0
$$
,we have  $\varphi_1^\perp(0, a_0)= 0$ is a solution of \eqref{5.2}. On the other hand,
$$
\displaystyle \frac{d}{d \varphi_1^\perp}|_{(\eta=0, a_0)}\Big(  \varphi_1^\perp - \eta (\Pi L_2)^{-1} \Pi F(a, \varphi_1^\perp)   \Big)
$$
is invertible for $0<|(\eta, a-a_0)| << 1$. Thus,  using the implicit  function theorem, there exists a unique $\varphi_1^\perp= \gamma(\eta, a)$ solving \eqref{5.2}.
Now we are going to solve \eqref{5.3} for $a= a(\eta)$.\\
We have $(F(a, \varphi_1^\perp), \varphi_1)_2=0$, which after expansion, leads to
$$
a|\varphi_1|_2^2 -a^3 |\varphi_1|_4^4 + O(\eta)=0,
$$
that in partcicular yields
\begin{equation}\label{5.4}
    -|\varphi_1|_2^2 -3a^2 |\varphi_1|_4^4 <0,
\end{equation}
by an appropriate choice of $|\eta|<<1$.\\
In summary, we can assert that there exists $0<\eta_0<<1$ such that for all $0< \eta<\eta_0$, there exist a unique solution $( \varphi_1^\perp, a)= (\varphi_1^\perp,a)(\eta))$ solving \eqref{5.2}-\eqref{5.3}.
finishing the proof of Proposition \ref{GS via bif}.\\
Second, we consider the eigenvalue problem of the linearized operator around $u(\eta)$

\begin{equation}\label{5.5}
    L_+ \varphi= \big(-\Delta + |x|^2 -3 (u(\eta))^2- \mu(\eta)\big) \varphi= \lambda \varphi,
\end{equation}
and track how the zero eigenvalue of $L_+$ moves for small values of $\eta$. 
Note that since $q$ (and hence $u$) decays exponentially fast in space, the operator $L_+$ has a discrete spectrum with the same asymptotic of the eigenvalues as the Harmonic oscillator.   Recalling that $\mu= \eta+2$ and $u(\eta)= \sqrt{\eta} q$ with $q=a \varphi_1+ \varphi_1^\perp$,  equation \eqref{5.5} is rewritten as
$$
L_\eta \varphi(\eta)= \lambda(\eta)  \varphi(\eta)
$$
with  $\| \varphi(\eta)\|_2^2=1$, and $L_\eta:= -\Delta + |x|^2 -3\eta q^2
-\eta-2$.\\
When $\eta=0$, we have $\lambda(0)=0$, $\varphi(0)=\varphi_1$, and $\lambda(0)$ is a simple isolated eigenvalue of $L_0$.\\
Taking the derivative with respect to $\eta$, we obtain at $\eta=0$:

\begin{equation}\label{5.6}
     \Big(-\Delta + |x|^2-2 -3( a(0)\varphi_1)^2-1\Big)+ L(0) \frac{d\varphi(0)}{d \eta}= \frac{d \lambda}{ d \eta} \varphi(0) + \lambda\frac{d  \varphi(0)}{ d \eta}.
\end{equation}
Multiplying the identity \eqref{5.6} by $\varphi(0)$ and taking the $L^2$ scalar product, we get

\begin{equation}\label{5.7}
\Big((-\Delta + |x|^2-2 -3( a(0)\varphi_1)^2-1)\varphi_1, \varphi_1\Big)_2 =-3 a^2(0)|\varphi_1|_4^4- |\varphi_1|_2^2.
\end{equation}
\eqref{5.7} is strictly negative by identity \eqref{5.4}. Therefore, $\displaystyle \frac{d \lambda}{d \eta}\Big|_{\eta=0}<0.$\\

Third, we show that for small masses $M$, the ground state $Q_M$ minimizing $V_M$  given by Lemma \ref{lemma 2} is indeed equal to the unique $u(\eta)$ given in proposition \eqref{GS via bif} by choosing $M\sim\eta<<\eta_0$. More precisely, it is sufficient to show that $\displaystyle \frac{Q_M}{\sqrt{M}}\rightarrow \varphi_1$ and $\mu_M\rightarrow 2$, as $M\rightarrow 0$ where $(-\Delta + |x|^2)\varphi_1= 2\varphi_1$ and $\|\varphi_1\|_2=1$. \\
To prove the latter assertion, let us first notice that

\begin{eqnarray}
\nonumber 2 M &\leq& \|\nabla Q_M\|_2^2 + \|x Q_M\|_2^2+\|Q_M\|_4^4 \\
\label{5.8}   &\leq& 2M +M^2 \|\varphi_1\|_4^4.
\end{eqnarray}
This implies that $\mu_M\rightarrow 2$ as $M\rightarrow 0$.\\
Additionally, \eqref{5.8} implies that $\displaystyle \frac{Q_M}{\sqrt{M}}=v_M$ is bounded in $\Sigma$ and satisfies

$$-\Delta v_M + |x|^2 v_M + M |v_M|^2 v_M= \mu_M v_M.$$Taking the weak limit in the latter inequality, we deduce that $v_M\rightarrow\varphi_1$ in $\Sigma$ as $M\rightarrow 0$. Hence, choosing $M$ small enough; $M\sim\eta<<M_0:=\eta_0$ so that $\frac{Q_M}{\sqrt M}$ is in a neighborhood of $\varphi_1$ in $\Sigma$.\\
In the last part of the proof, we choose
$$\alpha=  \frac{\displaystyle\int\sigma(x) Q_{M}^4(x) dx}{\displaystyle\int Q_{M}^2(x) dx}>\alpha_0:=  \frac{\displaystyle\int\sigma(x) Q_{M_0}^4(x) dx}{\displaystyle\int Q_{M_0}^2(x) dx} \sim \displaystyle\frac{1}{M_0}
$$ for $M<M_0$ yielding that $Q_M$ is a zero of the functional $\mathcal K$.\\\\

%
%
%
%
%
%
%
%
%

\end{proof}

\smallskip

\noindent \textbf{Acknowledgement: } S.I. was supported by  NSERC grant  (371637-2014). N. M  was partially supported by NSF grant    DMS-1716466.  The authors  would like to thank Peter Markowich for proposing them this problem, and are grateful to staff of King Abdullah University of Science and Technology for their great hospitality.

\bibliographystyle{jplain}

\Addresses

\end{document}